\documentclass{amsart}

\usepackage{amsfonts}
\usepackage{amssymb}

\newcommand{\rd}{\mathrm{d}}
\newcommand{\rn}{\mathbb{R}^{N}}
\newcommand{\R}{\mathbb{R}}
\newcommand{\A}{\mathcal{A}}
\newcommand{\J}{\mathcal{J}}
\newcommand{\I}{\mathcal{I}}
\renewcommand{\O}{\Omega}
\newcommand{\lp}{L^p(\Omega)}
\renewcommand{\H}{\mathcal{H}}
\newcommand{\lpb}{L^p(\partial\Omega)}

\newcommand{\lqb}{L^q(\partial\Omega)}
\renewcommand{\wp}{W^{1,p}(\Omega )}

\newcommand{\ep}{\varepsilon}

\newtheorem{de}{Definition}[section]
\newtheorem{lem}[de]{Lemma}
\newtheorem{te}[de]{Theorem}
\newtheorem{co}[de]{Corollary}
\newtheorem{pr}[de]{Proposition}
\newtheorem{ob}[de]{Remark}

\begin{document}

\title[Nonlinear elastic membranes]{Some optimization problems for nonlinear elastic membranes}
\author[L. Del Pezzo and J. Fern\'andez Bonder]
{Leandro M. Del Pezzo and Juli\'an Fern\'andez Bonder}

\address{Leandro M. Del Pezzo \hfill\break\indent
Departamento  de Matem\'atica, FCEyN, Universidad de Buenos Aires,
\hfill\break\indent Pabell\'on I, Ciudad Universitaria (1428),
Buenos Aires, Argentina.}

\email{{\tt ldpezzo@dm.uba.ar}}

\address{Juli\'an Fern\'andez Bonder \hfill\break\indent
Departamento  de Matem\'atica, FCEyN, Universidad de Buenos Aires,
\hfill\break\indent Pabell\'on I, Ciudad Universitaria (1428),
Buenos Aires, Argentina.}

\email{{\tt jfbonder@dm.uba.ar}\hfill\break\indent {\it Web page:}
{\tt http://mate.dm.uba.ar/$\sim$jfbonder}}

\begin{abstract}
In this paper we study some optimization problems for nonlinear elastic
membranes. More precisely, we consider the problem of optimizing the cost
functional $\J(u)=\int_{\partial\Omega} f(x) u\, \rd \H^{N-1}$ over some
admissible class of loads $f$ where $u$ is the (unique) solution to the problem
$-\Delta_p u + |u|^{p-2}u = 0$ in $\Omega$ with $|\nabla u|^{p-2}u_\nu = f$ on
$\partial \Omega$.
\end{abstract}

\maketitle

\section{Introduction}

In this paper we analyze the following optimization problem: Consider a smooth
bounded domain $\Omega \subset \rn$ and some class of admissible loads $\A$. Then
we want to maximize the cost functional
$$
\J(f) := \int_{\partial\Omega} f(x) u\, \rd \H^{N-1},
$$
for $f\in \A$, where $\H^d$ denotes the $d-$dimensional Hausdorff measure and
$u$ is the (unique) solution to the nonlinear membrane problem with load $f$
\begin{equation}\label{1.1}
\begin{cases}
-\Delta_p u + |u|^{p-2}u = 0 & \text{in } \O, \\
|\nabla u|^{p-2}\frac{\partial u}{\partial \nu} = f & \text{on }\partial\O.
\end{cases}
\end{equation}
Here, $\Delta_p u = \text{div}(|\nabla u|^{p-2}\nabla u)$ is the usual
$p-$Laplacian and $\frac{\partial}{\partial\nu}$ is the outer unit normal
derivative.

\medskip

These types of optimization problems have been considered in the literature due
to many applications in science and engineering, specially in the linear case
$p=2$. See for instance \cite{Cherkaev}.

In recent years, models involving the $p-$Laplacian operator with nonlinear
boundary conditions have been used in the theory of quasiregular and
quasiconformal mappings in Riemannian manifolds with boundary (see \cite{7,
16}), non-Newtonian fluids, reaction diffusion problems, flow through porus
media, nonlinear elasticity, glaciology, etc. (see \cite{1, 2, 3, 6}).

We want to stress that our results are new, even in the linear case. But since
our arguments are mainly variational, and for the sake of completeness, we
decided to present the paper in this generality.

\medskip

In this work, we have chosen three different classes of admissible functions
$\A$ to work with.
\begin{itemize}
\item The class of rearrangements of a given function $f_0$.

\item The (unit) ball in some $L^q$.

\item The class of characteristic functions of sets of given surface measure.
\end{itemize}

This latter case is what we believe is the most interesting one and where our
main results are obtained.

For each of these classes, we prove existence of a maximizing load (in the
respective class) and analyze properties of these maximizers.

The approach to the class of rearrangements follows the lines of \cite{cuccu},
where a similar problem was analyzed, namely, the maximization of the
functional
$$
\bar \J(g):= \int_{\Omega} g u\, \rd \H^N,
$$
where $u$ is the solution to $-\Delta_p u = g$ in $\Omega$ with Dirichlet
boundary conditions.

When we work in the unit ball of $L^q$ the problem becomes trivial and we
explicitly find the (unique) maximizer for $\J$, namely, the first
eigenfunction of a Steklov-like nonlinear eigenvalue problem (see Section 4).

Finally we arrive at the main part of the paper, namely, the class of
characteristic functions of sets of given boundary measure. In order to work
within this class, we first relax the problem and work with the weak* closure
of the characteristic functions (i.e. bounded functions of given $L^1$ norm),
prove existence of a maximizer within this relaxed class and then prove that
this optimizer is in fact a characteristic function. Then, in order to analyze
properties of this maximizer, we compute the first variation (or shape
derivative) with respect to perturbations on the set where the characteristic
function is supported.

This approach for optimization problems has been used several times in
the literature. Just to cite a few, see \cite{DPFBR, FBRW, Kurata} and
references therein. Also, our approach to the computation of the first
variation borrows ideas from \cite{GM}.

The paper is organized as follows. In Section 2 we include some preliminary
results, some of which are well known but we choose to include them in order to
make the paper self contained. In Section 3 we study the problem when the
admissible class of loads $\A$ is the class of rearrangements of a given
function $f_0$. In Section 4, we study the simpler case when $\A$ is the unit
ball in $L^q$. Finally, in Section 5, we analyze the case where $\A$ is the
class of characteristic functions of sets with given surface measure.

\section{Preliminaries}
\setcounter{equation}{0}

In this section we collect some well known results that will be used throughout
the paper.

\subsection{Results on rearrangements}

First, we recall some well known facts on rearrangements that will be needed in
Section 3.

\begin{de} Suppose $f:(X,\Sigma,\mu)\to\mathbb{R}^{+}$ and $g:(X',\Sigma',\mu')\to\mathbb{R}^{+}$
are measurable functions. We say $f$ and $g$ are rearrangements of
each other if and only if
$$
\mu(\{x\in X\colon f(x)\geq\alpha\})=\mu'(\{x\in X'\colon g(x)\geq\alpha\}), \quad
\forall\alpha\geq0.
$$
\end{de}

Now, given $f_0\in\ L^p(A)$, where $A\subset\rn$ with $\H^d(A)<\infty$, the set
of all rearrangements of $f_0$ is denoted by $\mathcal{R}_{f_0}$. Thus, for any
$f\in\mathcal{R}_{f_0},$ we have
$$
\H^d(\{x\in A:f(x)\geq\alpha\})=\H^d(\{x\in A:f_0(x)\geq\alpha\}), \quad
\forall\alpha\geq0.
$$

We will need the following Lemma, the proof of which can be found in
\cite{Burton}.

\begin{lem}\label{kempes} Let $f_0\in L^{p}(\partial\O)$ and $v\in
L^{p'}(\partial\O)$ such that $f_0,v\ge 0$. Then there exists
$\hat{f}\in\mathcal{R}_{f_0}$ such that
$$
\int_{\partial\O} \hat{f} v \, \rd \H^{N-1} =
\sup_{h\in\overline{\mathcal{R}}_{f_0}} \int_{\partial\O} h v \, \rd \H^{N-1}.
$$
\end{lem}

The following result can be easily deduced from \cite{ll} (Theorem 1.14 p.28).

\begin{te}[Bathtub Principle]\label{valdano}
Let $(\O,\Sigma,\mu)$ be a measurable space and let $f$ be a real-valued,
measurable function on $\O$ such that $\mu(\{ x:f(x)>t\})$ is finite for all
$t\in\mathbb{R}$. Let the number $G>0$ be given and define the class $\mathcal
C$ of measurable functions on $\O$ by
$$
\mathcal{C}=\big\{g\colon 0\leq g(x)\leq 1 \textrm{ for all } x \textrm{ and }
\int_\O g(x)\, \rd\mu = G\big\}.
$$
Then the maximization problem
$$
I=\sup_{g\in\mathcal{C}}\int_\O f(x)g(x)\, \rd\mu
$$
is solved by
\begin{equation}\label{lopez}
g(x)= \chi_{\{f>s\}}(x) + c \chi_{\{f=s\}}(x),
\end{equation}
where
$$
s=\inf\{t:\mu(\{f\geq t\})\leq G\}
$$
and
$$
c\mu(\{f=s\})=G-\mu(\{f>s\}).
$$
The maximizer given in \eqref{lopez} is unique if $G=\mu(\{f>s\})$ or if
$G=\mu(\{f\geq s\}).$
\end{te}

\subsection{Results on differential geometry}

Now we state without proof some results on differential geometry that will be
used in the last section. The proof of these results can be found, for
instance, in \cite{Henrot}.

\begin{de}[Definition of the tangential Jacobian]
Let $\Omega\subset\R^N$ be a smooth open set of $\R^N$. Let $\Phi$ be a $C^1$
field over $\R^N$. We call the tangential Jacobian of $\Phi$
$$
J_\tau(\Phi):= |^T[\Phi']^{-1}\nu| J(\Phi),
$$
where $\nu$ is the outer unit normal vector to $\partial\Omega$, $\Phi'$
denotes the differential matrix of $\Phi$, $J(\Phi)$ is the usual Jacobian of
$\Phi$ and $^TA$ is the transpose of the matrix $A$.
\end{de}

The definition of the tangential Jacobian is suited to state the
following change of variables formula

\begin{pr}\label{cambio.variables} Let $f\in L^1(\Phi(\partial\Omega))$. Then $f\circ \Phi\in
L^1(\partial\Omega)$ and
$$
\int_{\Phi(\partial\Omega)} f\, \rd \H^{N-1} = \int_{\partial\Omega} (f\circ
\Phi) J_\tau(\Phi)\, \rd \H^{N-1}.
$$
\end{pr}

\begin{de}[Definition of the tangential divergence]
Let $W$ be a $C^1$ vector field defined on $\R^N$. The tangential divergence of
$W$ over $\partial\Omega$ is defined as
$$
\rm{div}_\tau W := \rm{div} W - \langle W' \nu,\nu\rangle,
$$
where $\nu$ is the outer unit normal vector to $\partial\Omega$ and $\langle
\cdot,\cdot\rangle$ is the usual scalar product in $\R^N$.
\end{de}

With these definitions, we have the following version of the divergence
Theorem.

\begin{te}
Let $\Omega$ be a bounded smooth open set of $\R^N$, $D\subset\partial\Omega$
be a (relatively) open smooth set. Let $W$ be a $[W^{1,1}(\partial\Omega)]^N$
vector field. Then
$$
\int_D \rm{div}_\tau W\, \rd\H^{N-1} = \int_{\partial D} \langle W,
\nu_\tau\rangle\, \rd\H^{N-2} + \int_D H \langle W,\nu\rangle\, \rd\H^{N-1},
$$
where $\nu_\tau$ is the outer unit normal vector to $D$ along $\partial\Omega$
and $H$ is the mean curvature of $\partial\Omega$.
\end{te}

\section{Maximizing in the class of rearrangements}
\setcounter{equation}{0}

Given a domain $\O\subset\rn$ (bounded, connected, with smooth
boundary), first we want to study the following problem
\begin{equation}\label{bati}
\begin{cases}
-\Delta_p u + |u|^{p-2} u = 0 & \textrm{in }  \O,\\
|\nabla u|^{p-2}\frac{\partial u}{\partial\nu}= f & \textrm{on }
\partial\O.
\end{cases}
\end{equation}
Here $p\in(1,\infty)$, $\Delta_p u = {\rm div}(|\nabla u|^{p-2}\nabla u)$ is
the usual $p-$Laplacian, $\frac{\partial }{\partial\nu}$ is the outer normal
derivative and $f\in L^{q}(\partial \O)$ with $q>\frac{p'}{N'}$ .

We say $u\in\wp$ is a weak solution of \eqref{bati} if
$$
\int_\O |\nabla u|^{p-2}\nabla u\nabla v + |u|^{p-2}uv \, \rd \H^N =
\int_{\partial\O}fv \, \rd \H^{N-1}
$$
for all $v\in\wp$.

The restriction $q>\frac{p'}{N'}$ is related to the fact that $\frac{p'}{N'} =
p_*'$ where $p_* = p(N-1)/(N-p)$ is the critical exponent in the Sobolev trace
imbedding $W^{1,p}(\Omega)\hookrightarrow L^r(\partial\Omega)$. So, in order for that the right side of last equality to make sense for $f\in \lqb$ we need  $v$ to
belong to $L^{q'}(\Omega)$. This is achieved by the restriction $q'<p_*$.

It is a standard result that \eqref{bati}  has a unique weak solution $u_{f}$,
for which the following equations hold
\begin{equation}\label{rojas}
\int_{\partial\O}fu_{f} \, \rd \H^{N-1} = \sup_{u\in\wp} \I(u),
\end{equation}
where
$$
\I(u)=\frac{1}{p-1}\Big\{ p\int_{\partial\O} f u \, \rd \H^{N-1} - \int_\O
|\nabla u|^p + |u|^p \, \rd \H^N \Big\}.
$$

Let $f_0\in\lqb$, with $q=p/(p-1),$ and let $\mathcal{R}_{f_0}$ be the class of
rearrangements of $f_0$. We are interested in finding
\begin{equation}\label{cani}
\sup_{f\in \mathcal{R}_{f_0}} \int_{\partial\O} fu_{f} \, \rd \H^{N-1}.
\end{equation}

\begin{te}\label{messi}
There exists $\hat{f}\in \mathcal{R}_{f_0}$ such that
$$
\J(\hat f) = \int_{\partial\O} \hat{f}\hat{u} \, \rd \H^{N-1} = \sup_{f\in
\mathcal{R}_{f_0}} \J(f) = \sup_{f\in \mathcal{R}_{f_0}} \int_{\partial\O}
fu_{f} \, \rd \H^{N-1},
$$
where $\hat{u}=u_{\hat{f}}$.
\end{te}

\begin{proof}
Let
$$
I = \sup_{f\in \mathcal{R}_{f_0}} \int_{\partial\O} fu_{f} \, \rd \H^{N-1}.
$$
We first show that $I$ is finite. Let $f\in \mathcal{R}_{f_0}$. By
H\"{o}lder's inequality and the trace embedding we have
$$
\int_\O |\nabla u_{f}|^p + |u_{f}|^p \, \rd \H^N \leq C \|f\|_{\lqb}
\|u_{f}\|_{\wp},
$$
then
\begin{equation}\label{tevez}
\|u_{f}\|_{\wp}\leq C \quad \forall f\in \mathcal{R}_{f_0}
\end{equation}
since $\|f\|_{\lqb}=\|f_0\|_{\lqb}$ for all $f\in \mathcal{R}_{f_0}$. Therefore
$I$ is finite.

Now, let $\{f_i\}_{i\ge 1}$ be a maximizing sequence and let $u_i = u_{f_i}$.
From \eqref{tevez} it is clear that $\{u_i\}_{i\ge 1}$ is bounded in $\wp$,
then there exists a function $u\in\wp$ such that, for a subsequence that we
still call $\{u_i\}$,
\begin{eqnarray*}
u_i&\rightharpoonup&u \quad \textrm{weakly in } \wp,\\
u_i&\to&u \quad \textrm{strongly in } \lp,\\
u_i&\to&u \quad \textrm{strongly in } \lpb.
\end{eqnarray*}
On the other hand, since $\{f_i\}_{i\ge 1}$ is bounded in $\lqb$, we may choose
a subsequence, still denoted by $\{f_i\}_{i\ge 1}$, and $f\in\lqb$ such that
$$
f_i\rightharpoonup f \quad \textrm{weakly in } \lqb.
$$
Then
\begin{eqnarray*}
I & = & \lim_{i\to\infty} \int_{\partial\O}f_iu_i \, \rd \H^{N-1}\\
 & = & \frac{1}{p-1}\lim_{i\to\infty}\Big\{ p \int_{\partial\O} f_i u_i \, \rd \H^{N-1}
 - \int_\O |\nabla u_i|^p + |u_i|^p \, \rd \H^N \Big\}\\
 & \leq & \frac{1}{p-1}\Big\{ p\int_{\partial\O} f u \, \rd \H^{N-1}
 - \int_\O |\nabla u|^p + |u|^p \, \rd \H^N \Big\}.
\end{eqnarray*}
Furthermore, by Lemma \ref{kempes}, there exists $\hat{f}\in \mathcal{R}_{f_0}$
such that
$$
\int_{\partial\O} f u \, \rd \H^{N-1}\leq \int_{\partial\O} \hat{f} u \, \rd
\H^{N-1}.
$$
Thus
$$
I \leq \frac{1}{p-1}\Big\{ p\int_{\partial\O}\hat{f} u \, \rd \H^{N-1} -
\int_\O |\nabla u|^p + |u|^p \, \rd \H^N \Big\}.
$$
As a consequence of \eqref{rojas}, we have that
\begin{eqnarray*}
I & \leq & \frac{1}{p-1}\Big\{ p\int_{\partial\O}\hat{f} u \, \rd \H^{N-1}
- \int_\O |\nabla u|^p + |u|^p \, \rd \H^N \Big\}\\
 & \leq & \frac{1}{p-1}\Big\{ p\int_{\partial\O}\hat{f} \hat{u} \, \rd \H^{N-1}
- \int_\O |\nabla \hat{u}|^p + |\hat{u}|^p \, \rd \H^N \Big\}\\
 & = & \int_{\partial\O}\hat{f} \hat{u} \, \rd \H^{N-1}\\
 & \leq & I.
\end{eqnarray*}
Recall that $\hat{u}=u_{\hat{f}}.$ Therefore $\hat{f}$ is a solution to
\eqref{cani}. This completes the proof.
\end{proof}

\begin{ob}
With a similar proof we can prove a slighter stronger result. Namely, we can
consider the functional
$$
\J_1(f,g) := \int_\O gu\, \rd \H^N + \int_{\partial\O} f u\, \rd \H^{N-1},
$$
where $u$ is the (unique, weak) solution to
$$
\begin{cases}
-\Delta_p u + |u|^{p-2}u = g & \mbox{in }\Omega,\\
|\nabla u|^{p-2}\frac{\partial u}{\partial \nu} = f & \mbox{on }\partial\Omega,
\end{cases}
$$
and consider the problem of maximizing $\J_1$ over the class
$\mathcal{R}_{g_0}\times \mathcal{R}_{f_0}$ for some fixed $g_0$ and $f_0$.

We leave the details to the reader.
\end{ob}

\section{Maximizing in the unit ball of $L^q$}
\setcounter{equation}{0}

In this section we consider the optimization problem
$$
\max \J(f)
$$
where the maximum is taken over the unit ball in $L^q(\partial \Omega)$.

In this case, the answer is simple and we find that the maximizer can be
computed explicitly in terms of the extremal of the Sobolev trace embedding.

So, we let $f\in\lqb,$ with $q>\frac{p'}{N'}$, and $\|f\|_{\lqb}\leq 1$, we
consider the problem
\begin{equation}\label{delgado}
\sup_{f\in{\lqb}\atop\|f\|_{\lqb}\leq 1}\int_{\partial\O}fu_f\,\rd \H^N,
\end{equation}
where $u_f$ is the weak solution of
\begin{equation}\label{crespo}
\begin{cases}
-\Delta_p u + |u|^{p-2} u = 0 & \textrm{in }  \O,\\
|\nabla u|^{p-2}\frac{\partial u}{\partial\nu}= f & \textrm{on }
\partial\O.
\end{cases}
\end{equation}
The restriction $q>\frac{p'}{N'}$ is the same as in the previous section.

In this case it is easy to see that the solution becomes $\hat f =
v_{q'}^{q'-1}$ where $v_{q'}\in \wp$ is a nonnegative extremal for $S_{q'}$
normalized such that $\|v_{q'}\|_{L^{q'}(\partial\Omega)}=1$ and $S_{q'}$ is
the Sobolev trace constant given by
$$
S_{q'}=\inf_{v\in\wp}\frac{\int_\O |\nabla v|^p + |v|^p \, \rd \H^N}
{\big(\int_{\partial\O} |v|^{q'} \, \rd \H^{N-1}\big)^\frac{p}{q'}}.
$$
Furthermore $\hat u = u_{\hat f} = \frac{1}{S_{q'}^{1/p-1}} v_{q'}.$ Observe
that, as $q'<p_*$ there exists an extremal for $S_{q'}$. See \cite{FBR} and
references therein.

In fact
\begin{align*}
\J(\hat f) &= \int_{\partial \Omega} \hat f \hat u\, \rd \H^{N-1} =
\int_{\Omega} |\nabla \hat u|^p + |\hat u|^p\, \rd \H^N\\
&= \frac{1}{S_{q'}^{p/(p-1)}} \int_{\Omega} |\nabla v_{q'}|^p + |v_{q'}|^p\,
\rd \H^N = \frac{1}{S_{q'}^{1/(p-1)}}.
\end{align*}

On the other hand, given $f\in \lqb$, such that $\|f\|_{\lqb}\le 1$, we have
\begin{align*}
\J(f) &= \int_{\partial\Omega} f u_f\,\rd\H^{N-1} \le \|f\|_{\lqb}
\|u_f\|_{L^{q'}(\partial\Omega)} \\
&\le \Big( \frac{1}{S_{q'}} \int_{\Omega} |\nabla u_f|^p + |u_f|^p\,
\rd\H^N\Big)^{1/p} = \frac{1}{S_{q'}^{1/p}} \Big(\int_{\partial\Omega} f
u_f\,\rd\H^{N-1}\Big)^{1/p},
\end{align*}
{}from which it follows that
$$
\J(f)\le \frac{1}{S_{q'}^{1/(p-1)}}.
$$
This completes the characterization of the optimal load in this case.

\section{Maximizing in $L^\infty$}
\setcounter{equation}{0}

Now we consider the problem
\begin{equation}\label{cruz}
\sup_{\phi\in\mathbf{B}}\int_{\partial\O}\phi u_\phi\,\rd \H^{N-1},
\end{equation}
where $\mathbf{B}:=\{\phi : 0\leq\phi(x)\leq1\textrm{ for all } x\in\partial\O
\textrm{ and } \int_{\partial\O}\phi\,\rd \H^{N-1} = A\},$ for some fixed
$0<A<\H^{N-1}(\partial\O),$ and $u_{\phi}$ is the weak solution of

\begin{equation}\label{palermo}
\begin{cases}
-\Delta_p u + |u|^{p-2} u = 0 & \textrm{in }  \O,\\
|\nabla u|^{p-2}\frac{\partial u}{\partial\nu}= \phi & \textrm{on }
\partial\O.
\end{cases}
\end{equation}

This is the most interesting case considered in this paper.

\subsection{Existence of optimal configurations}

In this case, we have the following theorem:

\begin{te}\label{ferreyra} There exists $D\subset\partial\O$ with $\H^{N-1}(D)=A$ such that
$$
\int_{\partial\O}\chi_D u_D\,\rd \H^{N-1} = \sup_{\phi\in\mathbf{B}}
\int_{\partial\O}\phi u_\phi\,\rd \H^{N-1},
$$
where $u_D = u_{\chi_D}.$
\end{te}

\begin{proof}
Let
$$
I=\sup_{\phi\in\mathbf{B}}\int_{\partial\O}\phi u_\phi\,\rd \H^{N-1}.
$$
Arguing as in the first part of the proof for Theorem \ref{messi} we have that $I$
is finite.

Next, let $\{\phi_i\}_{i\ge 1}$ be a maximizing sequence and let $u_i =
u_{\phi_i}.$ It is clear that $\{u_i\}_{i\ge 1}$ is bounded in $\wp,$ then
there exists a function $u\in\wp$ such that, for a subsequence that we still
call $\{u_i\}_{i\ge 1}$
\begin{eqnarray*}
u_i&\rightharpoonup&u \quad \textrm{weakly in } \wp,\\
u_i&\to&u \quad \textrm{strongly in } \lp,\\
u_i&\to&u \quad \textrm{strongly in } \lpb.
\end{eqnarray*}
On the other hand, since $\{\phi_i\}_{i\ge 1}$ is bounded in
$L^\infty(\partial\O)$, we may choose a subsequence, again denoted
$\{\phi_i\}_{i\ge 1}$, and  $\phi\in L^\infty(\partial\O)$ and such that
$$
\phi_i\stackrel{*}{\rightharpoonup}\phi \quad \textrm{weakly* in } L^\infty(\partial\O).\\
$$
Then
\begin{eqnarray*}
I & = & \lim_{i\to\infty}\int_{\partial\O}\phi_i u_i \, \rd \H^{N-1}\\
& = & \frac{1}{p-1}\lim_{i\to\infty}\Big\{ p\int_{\partial\O} \phi_i u_i \, \rd
\H^{N-1} - \int_\O |\nabla u_i|^p + |u_i|^p \, \rd \H^N \Big\}\\
& \leq & \frac{1}{p-1}\Big\{ p\int_{\partial\O} \phi u \, \rd \H^{N-1} -
\int_\O |\nabla u|^p + |u|^p \, \rd \H^N \Big\}.
\end{eqnarray*}
Furthermore, by Theorem \ref{valdano}, there exists $D\subset \partial\O$ with
$\H^{N-1}(D)=A$ such that
$$
\int_{\partial\O} \phi u \, \rd \H^{N-1} \leq
 \int_{\partial\O} \chi_D u \, \rd \H^{N-1},
$$
and
$$
\{t< u\}\subset D\subset\{t\le u\}, \quad t:=\inf\{s:\H^{N-1}(\{s<u\})<A\}.
$$
Thus
$$
I\leq\frac{1}{p-1}\Big\{ p\int_{\partial\O}\chi_D u \, \rd \H^{N-1} - \int_\O
|\nabla u|^p + |u|^p \, \rd \H^N \Big\}.
$$
As a consequence of \eqref{rojas}, we have that
\begin{eqnarray*}
I & \leq & \frac{1}{p-1}\Big\{p\int_{\partial\O}\chi_D u \, \rd \H^{N-1} -
\int_\O |\nabla u|^p + |u|^p \, \rd \H^N \Big\}\\
& \leq & \frac{p}{p-1}\Big\{p\int_{\partial\O}\chi_D u_D \, \rd \H^{N-1} -
\int_\O |\nabla u_D|^p + |u_D|^p \, \rd \H^N \Big\}\\
& = & \int_{\partial\O}\chi_D u_D \, \rd \H^{N-1}\\
& \leq & I.
\end{eqnarray*}
Recall that $u_D=u_{\chi_D}.$ Therefore $\chi_D$ is a solution to \eqref{cruz}.
This completes the proof.
\end{proof}

\begin{ob}\label{marcico}
Note that in arguments in the proof of Theorem \ref{ferreyra}, using again the
Theorem \ref{valdano}, we can prove that
$$
\{t< u_D\}\subset D\subset\{t\le u_D\}
$$
where $t:=\inf\{s:\H^{N-1}(\{s<u_D\})<A\}.$ Therefore $u_D$ is constant on $\partial D.$
\end{ob}

\subsection{Domain Derivative}

In this subsection we compute the shape derivative of the functional
$\J(\chi_D)$ with respect to perturbations on the set $D$. We will consider
regular perturbations and assume that the set $D$ is a smooth subset of
$\partial \Omega$.

Then, by using the formula for the shape derivative, we deduce some necessary
conditions on a (regular) set $D$ in order for it to be optimal for $\J$ in the
$L^\infty$ setting.

Also, this formula could be used to derive algorithms in order to compute the
actual optimal set (cf. with \cite{FBGR}).

For the computation of the shape derivative, we use some ideas from \cite{GM}.

\medskip

We begin by describing the kind of variations that we are considering on the
set $D$. Let $V$ be a regular (smooth) vector field, globally Lipschitz, with
support in a neighborhood of $\partial\O$ such that $\langle V,\nu \rangle=0$
and let $\psi_t:\rn\to\rn$ be defined as the unique solution to
\begin{equation}\label{moreno}
\begin{cases}
\frac{\rd}{\rd t}\psi_t (x)=V(\psi_t(x)) & t>0,\\
\psi_0(x)= x & x\in \rn.
\end{cases}
\end{equation}

We have
$$
\psi_t(x)=x+tV(x)+ o(t) \quad \forall x\in\rn.
$$
Now, if $D\subset\partial\Omega$, we define $D_t :=
\psi_t(D)\subset\partial\O$.

\medskip

First, we compute the derivative at $t=0$ of the surface measure of the set
$D_t$. That is, we want to compute
$$
\frac{\rd}{\rd t} \H^{N-1}(D_t)\Big|_{t=0}.
$$

\begin{lem}\label{derivacion.area}
With the previous notation, if $D\subset \partial\O$ is a smooth (relatively)
open set, then
$$
\frac{\rd}{\rd t} \H^{N-1}(D_t)\Big|_{t=0} = \int_D \rm{div}V\, \rd\H^{N-1}.
$$
\end{lem}

\begin{proof}
We will use the following asymptotic formulae, for which the proofs can be found in
\cite{Henrot}:
\begin{align}
\label{jacobiano} J\psi_t(x) = 1 + t\,\rm{div} V(x) + o(t),\\
\label{inversa} [\psi_t^{-1}]'(x) = Id - tV(x) + o(t).
\end{align}

Then we have, by the change of variable formula, Proposition
\ref{cambio.variables},
$$
\H^{N-1}(D_t) = \int_{D_t} \rd \H^{N-1} = \int_D |[\psi_t^{-1}]'(x)\nu|
J\psi_t(x)\, \rd \H^{N-1}.
$$
Hence by \eqref{jacobiano}, \eqref{inversa} and the definition of $J_\tau$ we
get, using that $\langle V,\nu \rangle=0$,
$$
\H^{N-1}(D_t) = \H^{N-1}(D) + t \int_D \text{div}V\, \rd\H^{N-1} + o(t).
$$

Therefore, we arrive at
$$
\frac{\rd}{\rd t} \H^{N-1}(D_t)\Big|_{t=0}= \int_D \text{div}V\, \rd\H^{N-1}.
$$
This is what we wanted to show.\end{proof}

Now, let
$$
I(t)=\int_{\partial\O}u_t\chi_{D_t}\,\rd\H^{N-1},
$$
where $u_t\in \wp$ is the unique solution to
\begin{equation}\label{balbo}
\begin{cases}
-\Delta_p u_t + |u_t|^{p-2}u_t = 0 & \text{in }\O,\\
|\nabla u_t|^{p-2}\frac{\partial u_t}{\partial\nu} = \chi_{D_t} & \text{on
}\partial\O
\end{cases}
\end{equation}
and assume that $D\subset \partial\O$ is again a smooth (relatively) open set.

We have the following Lemma:
\begin{lem}\label{con}
Let $u_0$ and $u_t$ be the solution of \eqref{balbo} with $t=0$ and $t>0$,
respectively. Then
$$
u_t\to u_0 \textrm{ in } \wp, \textrm{ as } t\to 0^+.
$$
\end{lem}

\begin{proof}
The proof follows exactly as the one in Lemma 4.2 in \cite{cuccu}. The only
difference being that we use the trace inequality instead of the Poincar\'e
inequality.
\end{proof}

\begin{ob}\label{remark.conver}
It is easy to see that, as $\psi_t\to Id$ in the $C^1$ topology, then from
Lemma \ref{con} it follows that
$$
w_t:=u_t\circ \psi_t \to u_0 \qquad \text{strongly in } W^{1,p}(\O).
$$
\end{ob}

Now, we arrive at the main result of the section.

\begin{te}\label{teo.derivada}
With the previous notation, if $D\subset \partial\O$ is a smooth (relatively)
open set, we have that $I(t)$ is differentiable at $t=0$ and
$$
\frac{\rd}{\rd t}I(t)\Big|_{t=0} = \frac{p}{p-1}\int_{\partial D}u_0 \langle V,
\nu_\tau\rangle\,\rd\H^{N-2},
$$
where $u_0$ is the solution of \eqref{balbo} with $t=0$ and $\nu_\tau$ stands
for the exterior unit normal vector to $D$ along $\partial\Omega.$
\end{te}

\begin{proof}
By \eqref{rojas} we have that
$$
I(t)=\sup_{v\in\wp}\frac{1}{p-1}\bigg\{
p\int_{\partial\O}v\chi_{D_t}\,\rd\H^{N-1}-\int_\O |\nabla v|^p + |v|^p\,\rd
\H^N\bigg\}.
$$
Given $v\in\wp$ we consider $u=v\circ\psi_t\in\wp,$ then, by the change of
variables formula, Proposition \ref{cambio.variables},
\begin{align*}
\int_{\partial\O}v\chi_{D_t}\,\rd\H^{N-1} =&
\int_{\partial\O}u\chi_D J_\tau \psi_t\,\rd\H^{N-1}\\
= &
\int_{\partial\O}u\chi_D\,\rd\H^{N-1}+t\int_{\partial\O}u\chi_D\textrm{div}_\tau
V \,\rd\H^{N-1}+o(t).
\end{align*}

Also, by the usual change of variables formula, we have
\begin{align*}
\int_\O |\nabla v|^p\, \rd\H^N &=  \int_\O |^{T}[\psi_t']^{-1}(x) \nabla u^T|^p J\psi_t\,\rd \H^N\\
&=  \int_\O |(I-t ^T V'+ o(t)) \nabla u^T|^p \{1+t\textrm{div}V+o(t)\}\, \rd \H^N\\
&=  \int_\O\{|\nabla u|^p - tp|\nabla u|^{p-2}\langle\nabla u, ^T V'\nabla u^T\rangle +o(t)\}
\{1 + t\textrm{div}V + o(t)\}\, \rd \H^N\\
&=  \int_\O |\nabla u|^p \, \rd\H^N + t \int_\O |\nabla u|^p \textrm{div}V\, \rd\H^N\\
& - tp \int_\O |\nabla u|^{p-2} \langle\nabla u, ^T V'\nabla u^T\rangle \rd
\H^N +o(t),
\end{align*}
and
$$
\int_\O |v|^p\, \rd\H^N = \int_\O |u|^p J\psi_t\, \rd\H^N = \int_\O |u|^p\,
\rd\H^N + t \int_\O |u|^p\textrm{div}V\, \rd\H^N + o(t).
$$
Then, for all $v\in\wp$ we have that
\begin{align*}
p \int_{\partial\O}v\chi_{D_t}\, &\rd\H^{N-1} - \int_\O |\nabla v|^p + |v|^p\,\rd \H^N\\
= & p\int_{\partial\O}u\chi_{D}\,\rd\H^{N-1} - \int_\O |\nabla u|^p + |u|^p\,\rd \H^N\\
& + t\bigg[p\int_{\partial\O} u \chi_D\textrm{div}_\tau V\,\rd\H^N-\int_\O (|\nabla u|^p + |u|^p)
\textrm{div}V \, \rd\H^N\\
& + p\int_\O|\nabla u|^{p-2}\langle\nabla u, ^T V' \nabla u^T\rangle\, \rd
\H^N\bigg] + o(t).
\end{align*}
Therefore, we can rewrite $I(t)$ as
$$
I(t) = \sup_{u\in\wp}\frac{1}{p-1}\{\varphi(u) + t\phi(u) +o(t)\},
$$
where
$$
\varphi(u) = p\int_{\partial\O} u\chi_{D}\, \rd\H^{N-1} - \int_\O |\nabla u|^p
+ |u|^p\, \rd \H^N
$$
and
\begin{align*}
\phi(u) = & p\int_{\partial\O} u \chi_D\textrm{div}_\tau V\, \rd \H^{N-1} -
\int_\O (|\nabla u|^p + |u|^p) \textrm{div}V \, \rd\H^N\\
& + p \int_\O |\nabla u|^{p-2}\langle\nabla u, ^T V' \nabla u^T\rangle\, \rd
\H^N.
\end{align*}

If we define $w_t=u_t\circ\psi_t$ for all $t$ we have that $w_0 = u_0$ and
$$
I(t) = \frac{1}{p-1}\{\varphi(w_t) + t\phi(w_t) +o(t)\}
$$
for all $t$. Thus
$$
I(t)-I(0)\geq\frac{1}{p-1}\{ \varphi(u_0) + t\phi(u_0) + o(t)\} -
\frac{1}{p-1}\varphi(u_0),
$$
then
\begin{equation}\label{alsamendi}
\liminf_{t\to 0^+}\frac{I(t)-I(0)}{t}\geq\frac{1}{p-1}\phi(u_0).
\end{equation}

On the other hand
$$
I(t)-I(0)\leq\frac{1}{p-1}\{ \varphi(w_t) + t\phi(w_t) + o(t)\} -
\frac{1}{p-1}\varphi(w_t),
$$
hence,
$$
\frac{I(t)-I(0)}{t}\leq\frac{1}{p-1}\phi(w_t) + \frac{1}{t}o(t).
$$
By Remark \ref{remark.conver},
$$
\phi(w_t)\to\phi(u_0) \quad \textrm{as } t\to 0^+,
$$
therefore,
\begin{equation}\label{milito}
\limsup_{t\to 0^+}\frac{I(t)-I(0)}{t}\leq\frac{1}{p-1}\phi(u_0).
\end{equation}

From \eqref{alsamendi} and \eqref{milito} we deduced that there exists $I'(0)$
and
\begin{eqnarray*}
I'(0) & = & \frac{1}{p-1}\phi(u_0)\\
& = & \frac{1}{p-1} \bigg\{p\int_{\partial\O} u_0 \chi_D\textrm{div}_\tau V\,
\rd\H^{N-1} + p \int_\O |\nabla u_0|^{p-2}\langle\nabla u_0, ^T V' \nabla u_0^T\rangle\, \rd \H^N\\
& &\qquad \qquad - \int_\O (|\nabla u_0|^p + |u_0|^p)\textrm{div}V \,
\rd\H^N\bigg\}.
\end{eqnarray*}


Now we try to find a more explicit formula for $I'(0)$.

In the course of the computations, we require the solution $u_0$ to
$$
\begin{cases}
-\Delta u_0 + |u_0|^{p-2} u_0 = 0 & \mbox{in }\Omega,\\
|\nabla u_0|^{p-2}\frac{\partial u_0}{\partial \nu} = \chi_{D} & \mbox{on }
\partial\Omega,
\end{cases}
$$
to be $C^2$. However, this is not true. As it is well known (see, for instance,
\cite{16}), $u_0$ belongs to the class $C^{1,\delta}$ for some $0<\delta<1$.

In order to overcome this difficulty, we proceed as follows. We consider the
regularized problems
\begin{equation}\label{regularized}
\begin{cases}
-{\rm div}( (|\nabla u_0^\ep|^2 + \ep^2)^{(p-2)/2}\nabla u^\ep_0) + |u^\ep_0|^{p-2} u^\ep_0 = 0 & \mbox{in }\Omega,\\
(|\nabla u^\ep_0|^2 + \ep^2)^{(p-2)/2}\frac{\partial u^\ep_0}{\partial \nu} =
\chi_{D} & \mbox{on }\partial\Omega.
\end{cases}
\end{equation}
It is well known that the solution $u_0^\ep$ to \eqref{regularized} is of class
$C^{2,\rho}$ for some $0<\rho<1$ (see \cite{LSU}).

Then, we can perform all of our computations with the functions $u_0^\ep$ and
pass to the limit as $\ep\to 0+$ at the end.

We have chosen to work formally with the function $u_0$ in order to make our
arguments more transparent and leave the details to the reader. For a similar
approach, see \cite{GM}.

\medskip

Now, since
\begin{eqnarray*}
\textrm{div}(|u_0|^pV) & = & p|u_0|^{p-2}u_0\langle\nabla u_0, V\rangle + |u_0|^p\textrm{div}V,\\
\textrm{div}(|\nabla u_0|^pV)&=&p|\nabla u_0|^{p-2}\langle\nabla u_0 D^2 u_0,
V\rangle + |\nabla u_0|^p \textrm{div}V,
\end{eqnarray*}
we obtain
\begin{eqnarray*}
I'(0) & = & \frac{1}{p-1} \bigg\{p\int_{\partial\O} u_0 \chi_D\textrm{div}_\tau
V\,\rd\H^{N-1} + p\int_\O |\nabla u_0|^{p-2}\langle\nabla u_0, ^T V' \nabla u_0^T\rangle\, \rd \H^N\\
& & -\int_\O \textrm{div}((|\nabla u_0|^p + |u_0|^p)V)\, \rd\H^N
+ p\int_\O |\nabla u_0|^{p-2}\langle \nabla u_0 D^2 u_0, V\rangle\, \rd\H^N\\
& & + p\int_\O |u_0|^{p-2}u_0 \langle \nabla u_0, V\rangle \, \rd \H^N\bigg\}.\\
\end{eqnarray*}
Hence, using that $\langle V, \nu\rangle=0$ in the right hand side of the above
equality we find
\begin{eqnarray*}
I'(0) & = & \frac{p}{p-1}\bigg\{\int_{\partial\O} u_0 \chi_D\textrm{div}_\tau V\, \rd\H^{N-1}\\
& & + \int_\O |\nabla u_0|^{p-2} \langle \nabla u_0, ^T V' \nabla u_0^T + D^2 u_0 V^T\rangle\, \rd \H^N\\
& & + \int_\O |u_0|^{p-2}u_0 \langle \nabla u_0, V\rangle \, \rd \H^N\bigg\}\\
& = & \frac{p}{p-1} \bigg\{\int_{\partial\O} u_0 \chi_D\textrm{div}_\tau V\,
\rd\H^{N-1} + \int_\O |\nabla u_0|^{p-2} \langle \nabla u_0, \nabla(\langle\nabla u_0, V\rangle)\rangle\, \rd \H^N\\
& & + \int_\O |u_0|^{p-2}u_0 \langle \nabla u_0, V\rangle \, \rd \H^N\bigg\}.
\end{eqnarray*}
Since $u_0$ is a week solution of \eqref{balbo} with $t=0$ we have
\begin{eqnarray*}
I'(0) & = & \frac{p}{p-1}\bigg\{\int_{\partial\O} u_0 \chi_D\textrm{div}_\tau
V\, \rd\H^{N-1} + \int_{\partial\O} \langle \nabla u_0, V\rangle \chi_{D}\, \rd\H^{N-1}\bigg\}\\
& = & \frac{p}{p-1} \int_{\partial\O} \textrm{div}_\tau(u_0 V) \chi_D\, \rd\H^{N-1}\\
& = & \frac{p}{p-1} \int_{\partial D} u_0 \langle V,
\nu_\tau\rangle\,\rd\H^{N-2}.
\end{eqnarray*}
This completes the proof.
\end{proof}

The following corollary is a result that we have already observed, actually
under weaker assumptions on D, in Remark \ref{marcico}.

Nevertheless, we have chosen to include this remark as a direct application of
the Lemma \ref{derivacion.area} and Theorem \ref{teo.derivada}.

\begin{co}
Let $\chi_D$ be a maximizer for $\J$ over the class $\mathbf{B}$ and assume
that $D\subset \partial\Omega$ is a smooth (relatively) open set. Let $u_D$ be
the solution to the associated state equation
$$
\begin{cases}
-\Delta_p u + |u|^{p-2}u = 0 & \mbox{in }\Omega,\\
|\nabla u|^{p-2}\frac{\partial u}{\partial \nu} = \chi_D & \mbox{on
}\partial\Omega.
\end{cases}
$$
Then, $u_D$ is constant along $\partial D$.
\end{co}

\begin{proof}
Recalling the formula for the derivative of the volume, that is,
$$
\frac{\rd}{\rd t} \H^{N-1}(D_t)\Big|_{t=0}= \int_D \text{div}_\tau V\,
\rd\H^{N-1}=\int_{\partial D} \langle V, \nu_\tau\rangle \, \rd\H^{N-2},
$$
and the fact that $D$ is a critical point of $I,$ we derive
$$
I'(0)=c\frac{\rd}{\rd t} \H^{N-1}(D_t)\Big|_{t=0}\Longleftrightarrow
u=constant, \textrm{ on } \partial D.
$$
As we wanted to prove.
\end{proof}

\subsection{Final comments}

It would be interesting to say more about optimal configurations. For instance:
\begin{itemize}
\item What is the topology of optimal sets? Are optimal sets connected?

\item What about the regularity of optimal sets? Is it true that the boundary
of optimal sets are regular surfaces?

\item Where are the optimal sets located?
\end{itemize}

These questions, we believe, are difficult ones and we can only give an
answer in the trivial case where the domain $\Omega$ is a ball. In this case,
by symmetrization arguments (by means of the {\em spherical symmetrization},
cf. with \cite{FBRW, Sp}) it is straight forward to check that optimal sets are
spherical caps.

This example also shows that the uniqueness problem is far from obvious.

\end{document}